\documentclass[11pt]{amsart}
\usepackage{bm}
\usepackage{fullpage}
\usepackage{amssymb}
\usepackage{amsmath,amsfonts,amsthm}
\usepackage{hyperref}
\usepackage{color}
\usepackage{cite}
\usepackage{tikz}
\definecolor{uuuuuu}{rgb}{0.26666666666666666,0.26666666666666666,0.26666666666666666}

\newtheorem{theorem}{Theorem}[section]

\newtheorem{lemma}[theorem]{Lemma}

\newtheorem{claim}[theorem]{Claim}

\DeclareMathOperator\G{\mathcal{G}}

\DeclareMathOperator\F{\mathcal{F}}

\title{The $m$-bipartite Ramsey number $BR_m(K_{2,2},K_{5,5})$}
\author[1]{Yaser Rowshan$^1$  }

\address{$^1$Department of Mathematics, Institute for Advanced Studies in Basic Sciences (IASBS), Zanjan 66731-45137, Iran}

\email{y.rowshan@iasbs.ac.ir,~~~y.rowshan.math@gmail.com}

\begin{document}
	\maketitle
	
 	\begin{abstract} The bipartite Ramsey number $BR(H_1,H_2,\ldots,H_k)$ is the smallest positive integer $b$, such that each  $k$-decomposition of $E(K_{b,b})$ contains $H_i$ in the $i$-th class for some $i, 1\leq i\leq k$. As another view of bipartite Ramsey numbers, for the given two bipartite graphs $H_1$ and $H_2$ and a positive integer $m$, the $m$-bipartite Ramsey number $BR_m(H_1, H_2)$, is defined as the least integer $n$, such that any subgraph of  $K_{m,n}$ say $H$, results in $H_1\subseteq H$ or $H_2\subseteq \overline{H}$. The size of $BR_m(K_{2,2}, K_{3,3})$, $BR_m(K_{2,2}, K_{4,4})$ for each $m$, and  the size of   $BR_m(K_{3,3}, K_{3,3})$ for some $m$, have been determined in several papers up to now. Also, it is shown that $BR(K_{2,2}, K_{5,5})=17$.  In this article,  we compute the size  of  $BR_m(K_{2,2}, K_{5,5})$ for some $m\geq 2$.
	\end{abstract}
	
	\section{Introduction}
In his $1930$  on formal logic, F. Ramsey proved that if the $t$-combinations of an infinite class $\G$ are colored by $d$ distinct colors, then there exists a subclass $\F\subseteq \G$ so that all of the $t$-combinations of $\F$ have the same color. For $t = 2$, this is equivalent to saying that an infinite complete graph whose edges are colored in $d$ colors contains an infinite monochromatic complete subgraph. For given two graphs $G$ and $H$  the  Ramsey number	$R(G, H)$ is the minimum order of a complete graph such that any $2$-coloring of the edges must result in either a copy of graph $G$ in the first color or a copy of graph $H$ in the second color.  All such Ramsey numbers $R(G, H)$ exist as well. Also, it is shown that  $R(G, H)\leq R(K_m, K_n)$ where $|G|=m$ and $|H|=n$.
 
 Beineke and Schwenk, introduced the bipartite version of Ramsey numbers \cite{beinere1976bipartite}. For given bipartite graphs $H_1, H_2, \ldots H_k$, the bipartite Ramsey number $BR(H_1,H_2,\ldots,H_k)$  is the smallest positive integer $b$, such that each  $k$-decomposition of $E(K_{b,b})$, contains $H_i$ in the $i$-th class for some $i, 1\leq i\leq k$.  One can refer to \cite{goedgebeur2022new, gyarfas1973ramsey, hattingh1998star, raeisi2015star, bucic2019multicolour, bucic20193, lakshmi2020three, kamranian2022star, rowshan2023multicolor}, \cite{wang2021bipartite, hatala2021new, rowshan2022proof, gholami2021bipartite, goddard2000bipartite, rowshan2021size} and their references for further studies.
	 
	 Assume that $H_1$ and $H_2$ are two bipartite graphs. For each $m\geq 1$, the $m$-bipartite Ramsey number $BR_m(H_1, H_2)$, is defined as the least integer $n$, such that  any subgraph of  $K_{m,n}$ say $H$, results in $H_1\subseteq H$ or $H_2\subseteq \overline{H}$. The size of $BR_m(H_1, H_2)$ where $H_1\in \{K_{2,2}, K_{3,3}\}$ and  $H_2\in \{K_{3,3}, K_{4,4}\}$, have been determined in  some previous articles. In particular:
	\begin{theorem}\cite{bi2018another, 2022arXiv220112844R}
	For each  positive integer $m\geq 2 $, we have:
	\[
	BR_m(K_{2,2},K_{3,3})= \left\lbrace
	\begin{array}{ll}	
		\text{does not exist}, & ~~~where~~~m=2,3,\vspace{.2 cm}\\
		15& ~~~where~~~m=4,\vspace{.2 cm}\\
		12& ~~~where~~~m=5,6,\vspace{.2 cm}\\
		9 & ~~~where~~~ m=7,8.\vspace{.2 cm}\\
		
	\end{array}
	\right.
	\]	
\end{theorem}

	\begin{theorem}\cite{chartrand2021new, bi2019new}
	For each  positive integer $m\geq 2 $, we have:
	\[
	BR_m(K_{3,3},K_{3,3})= \left\lbrace
	\begin{array}{ll}	
		\text{does not exist}, & ~~~where~~m=2,3,4,\vspace{.2 cm}\\
		41 & ~~~where~~~m=5,6,\vspace{.2 cm}\\
		29 & ~~~where~~~ m=7,8.\vspace{.2 cm}\\
		
	\end{array}
	\right.
	\]	
\end{theorem} 
 
\begin{theorem}  \cite{rowshan2}
	For each  positive integer $m\geq 2 $, we have:
	\[
	BR_m(K_{2,2},K_{4,4})= \left\lbrace
	\begin{array}{ll}	
		\text{does not exist}, & ~~~~where~~m=2,3,4,\vspace{.2 cm}\\
		26 & ~~~where~~~m=5,\vspace{.2 cm}\\
		22 & ~~where~~~~ m=6,7,\vspace{.2 cm}\\
		16 & ~~~where~~~ m=8,\vspace{.2 cm}\\
		14 & ~~~where~~~ m\in \{9,10\ldots,13\}.\vspace{.2 cm}\\
	\end{array}
	\right.
	\]	
\end{theorem} 
	In this paper,  we compute the exact value of $BR_m(K_{2,2}, K_{5,5})$ for some $m\geq 2$ as follows.
	\begin{theorem}\label{M.th}[Main results]
		For each $m\in \{1,2,\ldots,8\}$, we have:
		\[
		BR_m(K_{2,2},K_{5,5})= \left\lbrace
		\begin{array}{ll}	
			\text{does not exist}, & ~~~where~~m=2,3,4,5,\vspace{.2 cm}\\
			40 & ~~~where~~~m=6,\vspace{.2 cm}\\
			30 & ~~~where~~~ m=7,8.\vspace{.2 cm}\\
		 
		\end{array}
		\right.
		\]	
	\end{theorem} 

	\section{Preparations}  Assume that $G[V, V']$ (or simply $[V, V' ]$), is a bipartite graph with bipartition sets $V$ and $V'$. Let $E(G[ W, W'])$, denotes the edge set of $G[ W, W']$. We use $\Delta (G_V)$ and $\Delta (G_{V'})$ to denote the maximum degree of  vertices in  part $V$ and $V'$ of $G$, respectively. The degree of a vertex $v\in V(G)$, is denoted by $\deg_G(v)$. For each $v\in V(V')$, $N_G(v) =\{ u\in V'(V), ~~vu\in E(G)\}$. For given  graphs $G$, $H$, and $F$, we say $G$ is $2$-colorable to $(H, F)$, if there is a  subgraph $G'$ of $G$, where $H\nsubseteq G'$ and $F\nsubseteq \overline{G'}$. We use $G\rightarrow (H, F)$, to show that  $G$  is $2$-colorable to $(H, F)$.  To simplify, we use $[n]$ instead of $\{1,,\ldots, n\}$.

	Alex F. Collins et al.,  have proven  the following theorem  \cite{collins2016zarankiewicz}.
	\begin{theorem}\label{th1}\cite{collins2016zarankiewicz} $BR(K_{2,2}, K_{5,5})=17$.
	\end{theorem}

	\begin{lemma}\label{l1} Suppose  that $(X=\{x_1,\ldots,x_m\},Y=\{y_1,\ldots ,y_{n}\})$, where $m\geq 6$ and $n\geq 10$ are the partition  sets of  $K=K_{m,n}$.	Let $G$ is a subgraph of $K_{m,n}$. If $\Delta (G_X)\geq 10$, then either $K_{2,2} \subseteq G$ or $K_{5,5} \subseteq \overline{G}$.
	\end{lemma}
	\begin{proof}
		Without loss of generality	(W.l.g), let $\Delta (G_X)=10$ and  $N_G(x)=Y'$, where  $|Y'|=10$ and $K_{2,2} \nsubseteq G$.  Therefore, 	 $|N_G(x')\cap Y'|\leq 1$ for each $x'\in X\setminus\{x\}$. Since $|X|\geq 6$ and $|Y'|= 10$,  one can check that  $K_{5,5} \subseteq \overline{G}[X\setminus\{x\}, Y']$.
	\end{proof}
	\section{\bf Proof of the main results}
To prove   Theorem \ref{M.th}, we need  the following theorems. 
\begin{theorem}\label{th0} For each $m\in \{2,3,4,5\}$, the number $BR_m(K_{2,2}, K_{5,5})$ does not exist.
 
\end{theorem}
\begin{proof}
Suppose that $m\in \{2,3,4,5\}$. For an arbitrary integer $t \geq 5$, set $K=K_{m,t}$ and let $G$ be a subgraph of $K$, such that $G\cong K_{1,t}$. Therefore, we have $\overline{G}\subseteq K_{m-1,t}$. Hence, neither  $K_{2,2}\subseteq G$ nor   $K_{5,5}\subseteq \overline{G}$. Which means that for each $m\in \{2,3,4,5\}$, the number $BR_m(K_{2,2}, K_{5,5})$ does not exist.
 
\end{proof}
	In the following theorem, we compute the size of $BR_m(K_{2,2}, K_{5,5})$ for $m=6$.
	\begin{theorem}\label{th2}
		$BR_6(K_{2,2}, K_{5,5})=40$.
	\end{theorem}
	\begin{proof}
		Suppose  that $(X=\{x_1,\ldots,x_6\},Y=\{y_1,y_2,\ldots ,y_{39}\})$ are the partition  sets of  $K=K_{6,39}$. Let $G\subseteq K$, such that   $N_G(x_i)=Y_i$ is as follows.
		\begin{itemize}
		 	
			\item{\bf (A1):} $Y_1=\{y_1,y_2,\ldots, y_9\}$.	
			\item{\bf (A2):} $Y_2=\{y_1,y_{10}, y_{11}\ldots, y_{17}\}$.
			\item{\bf (A3):} $Y_3=\{y_2,y_{10}, y_{18},y_{19}\ldots, y_{24}\}$.
			\item{\bf (A4):} $Y_4=\{y_3,y_{11}, y_{18}, y_{25},\ldots, y_{30}\}$.
			\item{\bf (A5):} $Y_5=\{y_4,y_{12}, y_{19}, y_{25}, y_{31}, y_{32},\ldots, y_{35}\}$.
			\item{\bf (A6):} $Y_6=\{y_5,y_{13}, y_{20}, y_{26}, y_{31}, y_{36},\ldots, y_{39}\}$.
		\end{itemize}
		Now, for each $i,j\in[6]$, by $(Ai)$ and $(Aj)$, it can be checked that  $|N_G(x_i)\cap N_G(x_j)|=1$,  and $|\cup_{j=1, j\neq i}^{j=6} N_G(x_j)|=35$. Therefore, $K_{2,2}\nsubseteq G$ and  $K_{5,5}\nsubseteq \overline{G}[X\setminus\{x_i\}, Y]$.  Which means that  $BR_5(K_{2,2}, K_{5,5})\geq 40$. 
		
		Now, suppose that $(X=\{x_1,\ldots, x_6\},Y=\{y_1,\ldots ,y_{40}\})$ are the partition  sets of  $K=K_{6,40}$. Suppose that $G$ is a subgraph of $K$, so that $K_{2,2} \nsubseteq G$. Consider $\Delta=\Delta (G_X)$. One can suppose that $\Delta\in \{8,9\}$. Otherwise  if $\Delta\geq 10$,  by Lemma \ref{l1} the proof is complete. And  if $\Delta\leq 7$, using the fact that $|Y|=40$,  it is clear that $K_{5,5}\subseteq \overline{G}$. Now, we have the following claims.
		\begin{claim}\label{c1}
			If	$ \Delta=8$, then  $K_{5,5}\subseteq \overline{G}$.
		\end{claim}	
		\begin{proof}[Proof of Claim \ref{c1}]
			W.l.g, let $Y_1=N_G(x_1)=\{y_1,\ldots,y_8\}$. As $K_{2,2} \nsubseteq G$,  we have $|N_G(x_i)\cap N_G(x_j)|\leq 1$ for each $i,j\in [6]$. Also, it can be checked that $|N_G(x)\cap Y_1|= 1$ for at least four members of $X\setminus \{x_1\}$, otherwise $K_{5,5}\subseteq \overline{G}[X,Y_1]$. W.l.g,  let $|N_G(x)\cap Y_1|= 1$ for each $x_i\in X'=\{x_2,x_3,x_4,x_5\}$.  If there is a member $x $ of $X'$, such that $|N_G(x)\cap Y\setminus Y_1|\leq 6$, then one can check that  $|\cup_{j=1}^{j=5} N_G(x_j)|\leq 35$. Hence, as $|Y|=40$,  it is clear that $K_{5,5}\subseteq \overline{G}[X\setminus\{x_6\}, Y]$.  So, let $|N_G(x)\cap (Y\setminus Y_1)|=7$ for each $x\in X'$. 	W.l.g, let $Y_2=N_G(x_2)=\{y_1, y_{9}, y_{11}\ldots,y_{15}\}$. Hence, there are at least two members $\{x_3,x_4\}$ of $X'\setminus \{x_2\}$, so that $|N_G(x_i)\cap Y_2\setminus \{y_1\}|= 1$ for each $i=3,4$, otherwise $K_{5,5}\subseteq \overline{G}[X\setminus \{x_2\},Y_2]$. Therefore, as 	$ \Delta=8$ and  $|N_G(x_j)\cap Y_i|= 1$ for each $i=1,2$ and $j=3,4$, we have $|\cup_{j=1}^{j=4} N_G(x_j)|\leq 27$.  Therefore, $|\cup_{j=1}^{j=5} N_G(x_j)|\leq 35$ and as $|Y|=40$,  we have $K_{5,5}\subseteq \overline{G}[X\setminus\{x_6\}, Y]$.  
		\end{proof}
So by Claim \ref{c1}, let $ \Delta=9$.   W.l.g let  $|N_G(x_1)|=9$ and $Y_1=N_G(x_1)=\{y_1,\ldots,y_9\}$. Now, we have the following claim.

	\begin{claim}\label{c2}
If either $|N_G(x_i)\cap Y_1|= 0$ or $N_G(x_i)\cap Y_1 = N_G(x_j)\cap Y_1$ for some $i,j\in\{2,\ldots,6\}$, then $K_{5,5}\subseteq \overline{G}$.
	\end{claim}	
	\begin{proof}[Proof of Claim \ref{c2}]
	 Since $K_{2,2} \nsubseteq G$, we have $|N_G(x_i)\cap Y_1|\leq 1$ for each $i$.  Now, let  $|N_G(x_2)\cap Y_1|= 0$.   It is clear that $K_{5,5}\subseteq \overline{G}[X\setminus\{x_1\},Y_1]$. Also, w.l.g let  $N_G(x_2)\cap N_G(x_3)\cap Y_1=\{y\}$, then as $|X|=6$ and  $|Y_1|=9$, we have $K_{5,5}\subseteq \overline{G}[X\setminus \{x_1\},Y_1\setminus \{y\}]$.	
	\end{proof}	
Therefore, by  Claim \ref{c2} and as $|Y_1=N_G(x_1)|=9=\Delta$, one can suppose that  $|N_G(x_i)\cap Y_1|= 1$ and  $N_G(x_i)\cap Y_1 \neq N_G(x_j)\cap Y_1$  for each $i, j\in \{2,3,4,5,6\}$. If  $|N_G(x)|=9$ for each $x\in X$, then by Claim \ref{c2}, it can be said that  $|\cup_{j=1}^{j=6} N_G(x_j)|= 39$. That is, there exists a member $y_{40}$ of $Y$, such that $|N_G(y_{40})|=0$.  Also,  by Claim  \ref{c2}, it is clear  that  $K_{5,4}\subseteq \overline{G}[X\setminus \{x_1\},Y_1]$. Therefore, we have $K_{5,5}\subseteq \overline{G}[X\setminus \{x_1\},Y_1\cup \{y_{40}\}]$ and the proof is complete.

 So, suppose that  $|N_G(x)|\leq 8$ for at least one member of  $X\setminus \{x_1\}$. 	W.l.g let $\deg_G(x_2)=8$ and $ Y_2=N_G(x_2)=\{y_1, y_{10}, y_{11}\ldots,y_{16}\}$. One can suppose that there exist at least three members of $X\setminus \{x_1,x_2\}$, say $\{x_3,x_4,x_5\}=X'$, so that $|N_G(x_i)\cap Y_2\setminus \{y_1\}|= 1$ and $N_G(x_i)\cap Y_2\setminus \{y_1\} \neq N_G(x_j)\cap Y_2\setminus \{y_1\}$ for each $i,j\in \{3,4,5\}$. Otherwise, by an argument similar to the proof of  Claim \ref{c2}, we have $K_{5,5}\subseteq \overline{G}[X\setminus \{x_2\},Y_2]$. If  $|N_G(x_i)\cap Y\setminus (Y_1\cup Y_2)|\leq 6$ for two members of $X'$, then  as $ \Delta=9$ and $|N_G(x_i)\cap Y_2\setminus \{y_1\}|= 1$ for each $x\in X'$, one can check that  $|\cup_{j=1}^{j=5} N_G(x_j)|\leq 35$. Hence, as $|Y|=40$,  we have $K_{5,5}\subseteq \overline{G}[X\setminus\{x_6\}, Y]$. So, w.l.g let $|N_G(x_i)\cap Y\setminus (Y_1\cup Y_2)|=7$ for $i=3,4$, that is  $|N_G(x_i)|= 9$ for $i=3,4$. Therefore, for $i=3,4$,   as $|N_G(x_i)|= \Delta$, by Claim \ref{c2}, we have $|N_G(x_3)\cap N_G(x_4)\cap Y\setminus (Y_1\cup Y_2)|=1$. Also, since $N_G(x_j)\cap Y_i\neq N_G(x_l)\cap Y_i$ for each $l,j\in \{3,4,5\}$ and $i=1,2$, one can check that $|\cup_{j=1}^{j=4} N_G(x_j)|= 29$. Now, consider $x_5$. Since $|N_G(x_i)|= 9$ for $i=3,4$, we have $| N_G(x_5)\cap N_G(x_i)|=1$ using Claim \ref{c2}. Therefore, as  $ \Delta=9$ and $|N_G(x_i)\cap N_{G}(x_5)|= 1$ for each $i=1,3,4$, one can say that $|N_G(x_5)\cap Y\setminus (Y_1\cup N_G(x_3)\cup N_G(x_4))|\leq 6$. So,   $|\cup_{j=1}^{j=5} N_G(x_j)|\leq 35$ and $K_{5,5}\subseteq \overline{G}[X\setminus\{x_6\}, Y]$ and the theorem holds.
\end{proof}
	In the following theorem, we compute the values of $BR_m(K_{2,2}, K_{5,5})$ for $m=7,8$.
	\begin{theorem}\label{th3}
		$BR_7(K_{2,2}, K_{5,5})=BR_8(K_{2,2}, K_{5,5})=30$.
	\end{theorem}
	\begin{proof}
		It is sufficient to show that:
		\begin{itemize}
			\item{\bf (I):}  $K_{8,29} \rightarrow (K_{2,2},K_{5,5})$.
			\item{\bf (II):} $BR_7(K_{2,2}, K_{5,5})\leq 30$.
		\end{itemize}
		We begin with $(I)$. Let $(X=\{x_1,\ldots, x_8\},Y=\{y_1,\ldots ,y_{29}\})$ are the partition  sets of  $K=K_{8,29}$. Suppose that $G$ is a subgraph of $ K$, such that for each $i\in[8]$,  $N_G(x_i)=Y_i$ is as follows.
		\begin{itemize}
			
			\item{\bf (B1):} $Y_1=\{y_1,\ldots, y_8\}$.	
			\item{\bf (B2):} $Y_2=\{y_1,y_9, y_{10},y_{11},y_{12}, y_{13}, y_{14}\}$.
			\item{\bf (B3):} $Y_3=\{y_2,y_9, y_{15},y_{16}, y_{17}, y_{18}, y_{19}\}$.
			\item{\bf (B4):} $Y_4=\{y_3,y_{10}, y_{15}, y_{20},y_{21}, y_{22}, y_{23}\}$.
			\item{\bf (B5):} $Y_5=\{y_4,y_{11}, y_{16}, y_{20}, y_{24}, y_{25},y_{26}\}$.
			\item{\bf (B6):} $Y_6=\{y_5,y_{12}, y_{17}, y_{21}, y_{24}, y_{27},y_{28}\}$.
			\item{\bf (B7):} $Y_7=\{y_6,y_{13}, y_{18}, y_{22}, y_{25}, y_{27},y_{29}\}$.
			\item{\bf (B8):} $Y_8=\{y_7,y_{14}, y_{19}, y_{23}, y_{26}, y_{28},y_{29}\}$.
		\end{itemize}
		Now, for each $i,j\in[8]$, by $(Di)$ and $(Dj)$, it can be said that:
		\begin{itemize}
			
			\item{\bf (C1):}  $|N_G(x_i)\cap N_G(x_j)|=1$, for each $i,j\in [8]$.
			\item{\bf (C2):} $|\cup_{i=1}^{i=5} N_G(x_{j_i})|= 25$, for each $j_1,j_2,j_3,j_4,j_5\in \{2,\ldots,8\}$.
			\item{\bf (C2):} $|\cup_{i=1}^{i=4} N_G(x_{j_i})\cup N_G(x_1)|= 26$, for each $j_1,j_2,j_3,j_4\in \{2,\ldots,8\}$.
		\end{itemize}
		
		By $(C1)$,  we have $K_{2,2}\nsubseteq G$. Also, by $(C2)$ and $(C3)$, one can check that  $K_{5,5}\nsubseteq \overline{G}$, which means that,  $K_{8,29} \rightarrow (K_{2,2},K_{5,5})$. Therefore, $(I)$ is correct.

		Now, we prove $(II)$. Assume that  $(X=\{x_1,\ldots,x_7\},Y=\{y_1,\ldots ,y_{30}\})$ are the partition  sets of  $K=K_{7,30}$. Let $G$ be a subgraph of $K$, where $K_{2,2} \nsubseteq G$. We show that $K_{5,5}\subseteq \overline{G}$.  Consider $\Delta=\Delta (G_X)$. Since $K_{2,2} \nsubseteq G$,  by Lemma \ref{l1}, one can assume that $\Delta\leq 9$.
		\begin{claim}\label{c3}
		If	$ \Delta= 9$, then $K_{5,5}\subseteq \overline{G}$.
		\end{claim}	
		\begin{proof}[Proof of Claim \ref{c3}]
		 W.l.g let $|N_G(x_1)=Y_1|=9$. Since $K_{2,2} \nsubseteq G$, $|X|=7$, and $|Y_1|=9$,  one can assume that $|N_G(x_i)\cap Y_1|= 1$, and $x_i$ and $x_j$ have  different neighborhood in $Y_1$ for each $i,j\in\{2,3,\ldots,7\}$. Otherwise, in any case $K_{5,5}\subseteq \overline{G}[X,Y_1]$. Therefore,   for each $x\neq x_1$, we have  $K_{5,4}\subseteq \overline{G}[X\setminus \{x_1,x\},Y_1]$.  So, if there is a member $y$ of $Y\setminus Y_1$, so that $|N_{\overline{G}}(y)\cap (X\setminus\{x_1\})|\geq 5$, then  $K_{5,5}\subseteq \overline{G}[X\setminus \{x_1\}, Y_1\cup \{y\}]$. Hence, let $|N_G(y)\cap (X\setminus\{x_1\}) |\geq 2$ for each $y\in Y\setminus Y_1$. So,  $|E(G[X, Y\setminus Y_1])|\geq 42$. Hence, by pigeon-hole principle, there is at least one member $x_2$ of $X\setminus\{x_1\}$, such that $|N_G(x_2)\cap (Y\setminus Y_1)|\geq 7$. Set $N_G(x_2)\cap (Y\setminus Y_1)=Y_2$. So, since $K_{2,2} \nsubseteq G$  for each  $i\in \{3,4,5,6,7\}$, we have $|N_G(x_i)\cap Y_2|\leq 1$. Therefore, as $|Y_2|\geq 7$, it is easy to say that $K_{5,1}\subseteq \overline{G}[X\setminus\{x_1,x_2\}, Y_2]$. Hence,  as $K_{5,4}\subseteq \overline{G}[X\setminus \{x_1,x_2\},Y_1]$, we have $K_{5,5}\subseteq \overline{G}[X\setminus\{x_1,x_2\}, Y_1\cup Y_2]$. So the claim holds.
		\end{proof}
		Now, by Claim \ref{c3}, let	$ \Delta\leq 8$. If $ \Delta\leq 5$, then it is clear that $K_{5,5}\subseteq \overline{G}$. 	So, one can suppose that $\Delta\in \{6,7,8\}$. Consider three cases as follows.
		
\bigskip 
\noindent \textbf{Case 1.} $ \Delta=6$.	W.l.g., suppose that $\Delta=|N_G(x_1)=Y_1|$. As $K_{2,2} \nsubseteq G$, one can assume that $|N_G(x_i)\cap Y_1|= 1$ and $N_G(x_i)\cap Y_1\neq N_G(x_j)\cap Y_1$ for at least three members of $X\setminus\{x_1\}$. Otherwise, $K_{5,5}\subseteq \overline{G}[X\setminus \{x_1\}, N_G(x_1)]$.  Set $X'=\{ x_2,x_3,x_4\}$, such that for each $x_i, x_j\in X'$ we have   $|N_G(x_i)\cap Y_1|= 1$  and $N_G(x_i)\cap Y_1\neq N_G(x_j)\cap Y_1$. W.l.g., assume that $Y_1=N_G(x_1)=\{y_1,\ldots, y_6\}$ and $x_iy_{i-1}\in E(G)$ for each $i=2,3,4$. If either  $|N_G(x)|\leq  4$ for one member of $X'$, or there exist two members $x_i, x_j$ of $X'$, such that $|N_G(x_i)=N_G(x_j)|=5$, then one can check that $|\cup_{i=1}^{i=5}N_G(x_i)|\leq 25$. Therefore, as $|Y|=30$, we have $K_{5,5}\subseteq \overline{G}$. So, suppose that  $|N_G(x_i)|=  6$ for at least two members of  $X'$. Now, w.l.g. assume that $|N_G(x_2)|=|N_G(x_3)|=6$ and $|N_G(x_4)|\geq 5$. W.l.g., assume that $Y_2=N_G(x_2)=\{y_1,y_7\ldots, y_{11}\}$.

 Let  $|N_G(x)\cap Y_2|=  1$ for one member of $\{x_3,x_4\}$, and w.lg let $x=x_3$, $Y_3=N_G(x_3)=\{y_2,y_7, y_{12}\ldots, y_{15}\}$. Hence, one can say that  $|\cup_{i=1}^{i=4}N_G(x_i)|\leq 20$. By considering  $x_5, x_6$, one can assume that either $|N_G(x)\cap ( \cup_{i=1}^{i=4}N_G(x_i))|\geq   1$ for at least one member of $\{x_5,x_6\}$ or $|N_G(x)|\leq 5$ for at least one member of $\{x_5,x_6\}$. Otherwise, we have $K_{2,2}\subseteq G[\{x_5, x_6\}, Y\setminus \cup_{i=1}^{i=4}N_G(x_i)]$ which is a concentration. Therefore, in any case, we have $|\cup_{i=1}^{i=4}N_G(x_i)\cup N_G(x)|\leq 25$ for some members of $\{x_5,x_6\}$ and the proof is complete. 
		
		So, suppose that $|N_G(x)\cap Y_2|=  0$ for each $x\in \{x_3,x_4\}$. W.l.g., assume that $Y_3=N_G(x_3)=\{y_2, y_{12}\ldots, y_{16}\}$. If either $|N_G(x_4)\cap Y_3\setminus \{y_2\}|=  1$  or $|N_G(x_3)|= 5$, then the proof is the same. So, suppose that $|N_G(x_4)\cap Y_3|=  0$  and $|N_G(x_3)|= 6$. W.l.g., assume that $Y_4=N_G(x_4)=\{y_3, y_{17}\ldots, y_{21}\}$. Hence, we have $|\cup_{i=1}^{i=4}N_G(x_i)|=21$. By considering  $x_5, x_6$, if either $|N_G(x)\cap ( \cup_{i=1}^{i=4}N_G(x_i))|\geq   2$ for at least one member of $\{x_5,x_6\}$,  or $|N_G(x)|\leq 4$ for at least one member of $\{x_5,x_6\}$, then we have $|\cup_{i=1}^{i=4}N_G(x_i)\cup N_G(x)|\leq 25$ for some members of $\{x_5,x_6\}$ and the proof is complete. So, for each member of $\{x_5,x_6\}$, suppose that $|N_G(x)|\geq 5$ and $|N_G(x)\cap ( \cup_{i=1}^{i=4}N_G(x_i))|\leq   1$. Therefore, as $|N_G(x)\cap ( \cup_{i=1}^{i=4}N_G(x_i))|\leq   1$, by pigeon-hole principle, one can assume that there exists at least one $i\in \{2,3,4\}$ say $i=2$, such that $|N_G(x)\cap  N_G(x_2)|=0$ for each $x\in \{x_5,x_6\}$. Hence, as   $|N_G(x_2)|=6$ and $|N_G(x_2)\cap  N_G(x_i)|=0$ for $i=3,4,5,6$, one can say that   $K_{5,5}\subseteq \overline{G}[X\setminus \{x_2\}, Y_2]$. Therefore, the proof of Case 1 is complete.
	
\bigskip 
	\noindent \textbf{Case 2.} $ \Delta=7$. W.l.g., assume that $ N_G(x_1)=Y_1=\{y_1,y_2,\ldots,y_7\}$. As $K_{2,2} \nsubseteq G$, we have $|N_G(x_i)\cap Y_1|\leq 1$. Hence,  one can say that $K_{5,2}\subseteq \overline{G}[X\setminus \{x_1,x\},Y_1]$ for each $x\in X\setminus\{x_1\}$. Now, set $B$ as follows:
	\[B=\{x\in X, ~~|N_G(x)|=7\}\]
	By considering $B$, we have the following claim.
	
\begin{claim}\label{c6}
If $|B|= 1$, then  $K_{5,5}\subseteq \overline{G}$.
\end{claim}	
\begin{proof}[Proof of Claim \ref{c6}]
 Set $X_1=X\setminus \{x_1\}$. As $|Y_1|=7$ and $|N_G(x_i)\cap Y_1|\leq 1$,  one can say that $|N_G(x)\cap Y_1|= 1$ for at least four vertices of $X_1$. Otherwise,
 $K_{5,5}\subseteq \overline{G}$. Therefore, set $X'=\{x\in X_1, ~~|N_G(x)\cap Y_1|= 1\}$ and w.l.g. let $\{x_2,x_3,x_4,x_5\}\subseteq X'$. If $|\cup_{j=2}^{j=5} N_G(x_j)\cap (Y\setminus Y_1)|\leq 18 $,  then we have $|\cup_{j=1}^{j=5} N_G(x_j)|\leq 25 $ and so	$K_{5,5}\subseteq \overline{G}$. Now, suppose that $|\cup_{j=2}^{j=5} N_G(x_j)\cap (Y\setminus Y_1)|\geq 19 $, which means that there exist at least three vertices $\{x_2,x_3,x_4\}=X''$ of $X'$, such that $|N_G(x_j)\cap (Y\setminus Y_1)|=5$ and $N_G(x_i)\cap  N_G(x_j)\cap (Y\setminus Y_1) =\emptyset$ for each $i,j\in \{2,3,4\}$. W.l.g., assume that $Y_2=\{y_8,y_9\ldots,y_{12}\}$, $ Y_3=\{y_{13},y_{14}\ldots,y_{17}\}$, and $Y_4=\{y_{18},y_{19},\ldots,y_{22}\}$, in which $Y_i=N_G(x_i)\cap (Y\setminus Y_1)$ for $i=2,3,4$. Also, since $|\cup_{j=2}^{j=5} N_G(x_j)\cap (Y\setminus Y_1)|\geq 19 $ and $x_5\in X'$, we can assume that $Y_5=\{y_{23},y_{24},\ldots,y_{26}\}\subseteq N_G(x_5)\cap (Y\setminus Y_1)$. Therefore, as  $|N_G(x_5)\cap (Y\setminus Y_1)|\leq 5$, one can say that $|N_G(x_5)\cap (Y\setminus Y_1) \cap Y_i|=0$ for at least two $i$, which $i\in \{2,3,4\}$. W.l.g., assume that for $i=2,3$, we have $|N_G(x_5)\cap Y_i|=0$. Therefore, if there exists one $j\in \{6,7\}$  and one $i\in \{2,3\}$, such that $|N_G(x_j)\cap Y_i|=0$, then  $K_{5,5}\subseteq \overline{G}[X''\setminus \{x_i\}\cup \{x_1,x_5, x_j \}, Y_i]$.   Therefore, suppose that $|N_G(x_j)\cap Y_i|=1$ for each $i\in \{2,3\}$ and each $j\in \{6,7\}$. Hence, as $|N_G(x_j)\cap (Y\setminus Y_1)|\leq 6$ for $j=6,7$, we have $|\cup_{x\in X'''}N_G(x_j)|\leq 24$, where  $X'''=\{x_1,x_2,x_3, x_6,x_7\}$. Which means that  $K_{5,5}\subseteq \overline{G}[X''', Y\setminus  \cup_{x\in X'''}N_G(x_j)]$ and the claim holds.
\end{proof}	
  
Now, by Claim \ref{c6}, we may suppose that $|B|\geq 2$. Hence, we have the following claim.
\begin{claim}\label{c7}
If there exist $x, x'\in B$, so that $|N_G(x)\cap N_G(x')|=0$, then $K_{5,5}\subseteq \overline{G}$. 
\end{claim}	
\begin{proof}[Proof of Claim \ref{c7}]
W.l.g., assume that $x_1, x_2\in B$, $|N_G(x_1)\cap N_G(x_2)|=0$, $ N_G(x_1)=Y_1=\{y_1,\ldots,y_7\}$, and $ N_G(x_2)=Y_2=\{y_8,\ldots,y_{14}\}$. Since $K_{2,2} \nsubseteq G$, so for $i=1,2$, we have $|N_G(x_i)\cap Y_i|\leq 1$ and since $|Y_i|=7$, we have $K_{5,4}\subseteq \overline{G}[X\setminus \{x_1,x_2\},Y_1\cup Y_2]$. Also, we may suppose that the following two facts are established. Otherwise, in any case it is easy to check that $K_{5,5}\subseteq \overline{G}$.
\begin{itemize}
	\item {\bf (F1):} $|N_G(x_j)\cap Y_i|=1$ and $N_G(x_j)\cap Y_i\neq N_G(x_l)\cap Y_i$, for each $j, l\in\{3,4,\ldots,7\}$ and  $i\in \{1,2\}$. 
	\item {\bf (F2):} For each $j\in\{3,4,\ldots,7\}$ suppose that $Y_j=N_G(x_j)\cap (Y\setminus Y_1\cup Y_2)$, hence for each $j_1,j_2,j_3\in \{3,4,5,6,7\}$, we have 
$|\cup_{l=1}^{l=3} N_G(x_{j_l})|\geq 12 $.	
\end{itemize}
By $(F1)$ and $(F2)$, we can prove the following fact.
\begin{itemize}
	 	
	\item {\bf (F3):}  If there exists $j\in \{3,\ldots,7\}$, such that $|Y_j=N_G(x_j)\cap (Y\setminus Y_1\cup Y_2)|=5$, then  for each $x \in X\setminus \{x_1,x_2,x_j\}$, we have $| N_G(x)\cap Y_j|=1$.	
\end{itemize}
W.l.g., assume that $|Y_3=N_G(x_3)\cap (Y\setminus Y_1\cup Y_2)|=5$. Therefore, if there exists  $x\in X\setminus \{x_1,x_2,x_3\}$, such that $| N_G(x)\cap Y_3|=0$, then one can say that  $K_{5,2}\subseteq \overline{G}[ \{x_1,x_4,x_5,x_6,x_7\},\{y,y'\}]$ for some $y,y'\in Y_3$. Hence, as $|N_G(x_1)\cap N_G(x_2)|=0$, $|N_G(x_2)|=7$, and $|N_G(x)\cap N_G(x_2)|=1$ for each $x\in X\setminus \{x_1,x_2\}$, it is easy to say that $K_{5,3}\subseteq \overline{G}[ \{x_1,x_4,x_5,x_6,x_7\},Y_2]$, which means that $K_{5,5}\subseteq \overline{G}[ \{x_1,x_4,x_5,x_6,x_7\}, Y_2\cup Y_3]$. Hence $(F3)$ is true.

Consider $|B|$ and suppose that $|B|\leq 3$.  By $(F1)$ and $(F2)$, we have  $|Y_j=N_G(x_j)\cap (Y\setminus Y_1\cup Y_2)|=4$ for at least four $j\in \{3,4,\ldots,7\}$. W.l.g.,  suppose that $M=\{3,4,5,6\}$ and $|Y_j=N_G(x_j)\cap (Y\setminus Y_1\cup Y_2)|=4$ for each $j\in M$. Therefore, for each $j_1,j_2,j_3\in M$, we have $|\cup_{l=1}^{l=3} N_G(x_{j_l})|= 12 $, which means that $|Y_{j_1}\cap Y_{j_2}|=0$ for each $j_1, j_2\in M$. Therefore, as  $|Y\setminus( Y_1\cup Y_2)|=16$ and $|X\setminus \{x_1,x_2\}|=5$, by considering $x_7$, if $|N_G(x_7)|=7$, then $K_{2,2}\subseteq G$    which is a contradiction.  Also, if $|N_G(x_7)|\leq 6$, then $|\cup_{l=1}^{l=2} N_G(x_{j_l})\cup N_G(x_7)|= 11$ for some $j_1,j_2\in M$, which is a contradiction with $(F2)$.

So, suppose that $|B|\geq 4$ and w.l.g. assume that $x_3, x_4\in B$. Therefore, by $(F1)$ and $(F3)$, we  can suppose  that $Y_3=\{y_{15}, y_{16}, y_{17}, y_{18}, y_{19}\}$ and  $Y_4=\{y_{15}, y_{20}, y_{21}, y_{22}, y_{23}\}$, where $Y_i=N_G(x_i)\cap (Y\setminus Y_1\cup Y_2)$ for $i=3,4$. Now, by $(F2)$ and $(F3)$, we have $B=X$. Otherwise, w.l.g. let $x_5\notin B$, using $(F3)$ we have $|N_G(x_5)|\leq 6$. Hence, by $(F1)$, $(F3)$,  we have $|N_G(x_5)\cap (Y\setminus \cup_{i=1}^{i=4} Y_i)	|\leq 2$. So, we can say that $| \cup_{i=3}^{i=5} Y_i|\leq 11$, a contradiction with $(F2)$. Therefore, by $(F1), (F3)$, w.l.g. assume that  $Y_5=\{y, y', y_{24}, y_{25}, y_{26}\}$, where $y\in Y_3$ and $y'\in Y_4$. Also by $(F1), (F3)$, w.l.g. let   $Y_6=\{y'', y''', y_{24}, y_{27}, y_{28}\}$, where $y''\in Y_3$ and $y'''\in Y_4$. Therefore, it is easy to say that $| \cup_{i=2}^{i=6} Y_i|\leq 25$, which means that $K_{5,5}\subseteq \overline{G}[ \{x_2, x_3,x_4,x_5,x_6\},Y\setminus \cup_{i=2}^{i=6} Y_i]$ and the claim holds.
\end{proof}	
Now, by considering Claim \ref{c7}, we verify the following claim.
 \begin{claim}\label{c8}
 	If $|B|= 2$, then we have $K_{5,5}\subseteq \overline{G}$.
 \end{claim}	
 \begin{proof}[Proof of Claim \ref{c8}]
 	Suppose that $x_1,x_2\in B$ and $Y_i=N_G(x_i)$ for $i=1,2$. By Claim \ref{c7}, w.l.g. let $Y_2=\{y_1, y_8, y_9, \ldots, y_{13}\}$. Set $X_1=X\setminus \{x_1,x_2\}$. Suppose that $Y_j=N_G(x_j)\cap (Y\setminus Y_1\cup Y_2)$, hence for each $j_1,j_2,j_3\in \{3,4,5,6,7\}$, we may assume that $|\cup_{l=1}^{l=3} N_G(x_{j_l})|\geq 13 $; otherwise $K_{5,5}\subseteq \overline{G}$. Therefore, as $|X_1|=5$, it is easy to check that there exist at least three vertices $\{x_3,x_4, x_5\}=X'_1$ of $X_1$, such that $|Y_j=N_G(x_j)\cap (Y\setminus Y_1\cup Y_2)|=5$. Hence, we have 	$|N_G(x_j)\cap (Y_1\cup Y_2)|\leq 1$  for each $x_j\in X_1'$. Now, consider $x_6, x_7$, as   $K_{2,2} \nsubseteq G$, so for $i=1,2$ and $j=6,7$, we have $|N_G(x_j)\cap Y_i|\leq 1$. Therefore as $|Y_1\cup Y_2|=13$, we have $|\cup_{j=3}^{j=7} (N_G(x_j)\cap (Y_1\cup Y_2))|\leq 7$, which means that  $K_{5,5}\subseteq \overline{G}[X_1, Y_1\cup Y_2]$ and the claim holds.	 
 \end{proof}	
 Therefore, by  Claims \ref{c6}, \ref{c8}, let $|B|\geq 3$ and w.l.g.  suppose that $\{x_1,x_2,x_3\}\subseteq B$. By considering the members of $B$, we prove the following claim. 
 \begin{claim}\label{c0}
 	If  $|N_G(x)\cap N_G(x')\cap N_G(x'')|=1$ for some  $x,x',x'' \in B$, then $K_{5,5}\subseteq \overline{G}$. 
 \end{claim}	
 \begin{proof}[Proof of Claim \ref{c0}]
 	W.l.g., assume that $\{y_1\}=N_G(x_1)\cap N_G(x_2)\cap N_G(x_3) $, $Y_1=\{y_1,\ldots,y_7\}$, $Y_2=\{y_1,y_8,\ldots,y_{13}\}$, and  $Y_3=\{y_1, y_{14},\ldots,y_{19}\}$, where $Y_i=N_G(x_i)$ for $i=1,2,3$. Since $K_{2,2} \nsubseteq G$,  for $i=1,2,3$ and each $x\in X\setminus\{x_1,x_2,x_3\}$, we have $|N_G(x)\cap Y_i|\leq 1$. Now, we have the following  fact.
 	\begin{itemize}
 		\item {\bf (P1):} For each $j, l\in\{4,\ldots,7\}$ and for each $i\in \{1,2,3\}$,  $|N_G(x_j)\cap Y_i\setminus \{y_1\}|=1$ and $(N_G(x_j)\cap Y_i\setminus \{y_1\})\neq ( N_G(x_l)\cap Y_i\setminus\{y_1\})$.		
 	\end{itemize}
 Suppose that $|N_G(x)\cap  Y_i\setminus\{y_1\}|=0$ for at least one $x$ and at least one $i\in \{1,2,3\}$. W.l.g., let $|N_G(x_4)\cap  Y_1\setminus \{y_1\}|=0$. Therefore, as $|N_G(x)\cap Y_i|\leq 1$, one can say that $K_{2,11}\subseteq \overline{G}[\{x_3,x_4\}, Y_1\cup Y_2\setminus \{y_1\}]$.  Since  $|N_G(x)\cap Y_i|\leq 1$ for $i=5,6,7$, it is easy to say that $|\cup _{j=5}^{j=7}(N_G(x_j)\cap (Y_1\cup Y_2\setminus \{y_1\}))|\leq 6$. Therefore, one can say that  $K_{3,6}\subseteq \overline{G}[\{x_5,x_6,x_7\}, Y_1\cup Y_2\setminus \{y_1\}]$. Hence, we have $K_{5,5}\subseteq \overline{G}[\{x_3,x_4, x_5,x_6,x_7\}, Y_1\cup Y_2\setminus \{y_1\}]$. By an argument similar to the proof of  $|N_G(x_j)\cap Y_i\setminus \{y_1\}|=1$,  one can  check that $(N_G(x_j)\cap Y_i\setminus \{y_1\})\cap ( N_G(x_l)\cap Y_i\setminus\{y_1\})=\emptyset$ for each $j,l\in \{4,\ldots,7\}$. So, the fact is true.
 
Also, as $| \cup_{i=1}^{i=3} Y_i|= 19$, we can say that the following  fact is true; otherwise  it is easy to check that $K_{5,5}\subseteq \overline{G}$.
 \begin{itemize}
 \item {\bf (P2):} For each $j\in\{4,\ldots,7\}$, let $Y_j=N_G(x_j)\cap (Y\setminus Y_1\cup Y_2\cup Y_3)$. Then for each $j_1,j_2\in \{3,4,5,6,7\}$, we have $|\cup_{l=1}^{l=2} Y_{j_l}|\geq 7$.	
 \end{itemize}

Therefore, by $(P2)$, one can assume  that $|Y_j|=4$ for at least three $j\in \{4,5,6,7\}$. W.l.g., let $|Y_j|=4$ for each $j\in \{4,5,6\}$. Hence by $(P1)$, we have $x_j\in B$ for each $j\in \{4,5,6\}$. W.l.g., assume that $Y_4=\{y_{20}, y_{21}, y_{22}, y_{23}\}$. Hence, by $(P2)$ and  Claim \ref{c7}, one can assume that $|Y_4\cap Y_5|=1$. So, let $Y_5=\{y_{20}, y_{24}, y_{25}, y_{26}\}$. Now consider $Y_6$. Since  $|Y_i\cap Y_6|=1$, if $y_{20}\notin Y_6$,  it is easy to say that $|Y_4\cup Y_5\cup Y_6|=9$, moreover in this case by $(P1)$, we have $|Y_2\cup Y_3\cup Y_4\cup Y_5\cup Y_6|=25$, which means that $K_{5,5}\subseteq \overline{G}[\{x_2,\ldots,x_6\}, Y]$. Therefore, let  $y_{20}\in Y_6$ and  $Y_6=\{y_{20}, y_{27}, y_{28}, y_{29}\}$. Now, consider $x_7$, if $y_{30}\notin N_G(x_7)$, then  $|Y_{j}\cup  Y_7|\leq 6$ for each $j\in \{4,5,6\}$ which is a contradiction with $(P2)$. So, suppose that $y_{30}\in N_G(x_7)$. Since $\deg_G(x_7)\leq 7$,  if $y_{20}\in N_G(x_7)$, then    $|Y_{j}\cup  Y_7|\leq 5$ for each $j\in \{4,5,6\}$, a contradiction to $(P2)$. Therefore, assume that   $y_{20}\notin N_G(x_7)$. Now, we can assume that $|N_G(x_7)\cap Y_i\setminus\{y_{20}\}|\neq 0$ for each $i=4,5,6$. Otherwise, w.l.g. let $|N_G(x_7)\cap Y_4\setminus\{y_{20}\}|= 0$, then  $|Y_5\cup Y_6\cup Y_7|=8$, moreover in this case, by $(P1)$, we have $|Y_2\cup Y_3\cup Y_5\cup Y_6\cup Y_7|\leq 25$, which means that $K_{5,5}\subseteq \overline{G}$. So, suppose that  $|N_G(x_7)\cap Y_i\setminus\{y_{20}\}|\neq 0$. W.l.g., let $Y_7=\{y_{21}, y_{24}, y_{27}, y_{30}\}$. Hence,  $|Y_4\cup Y_5\cup Y_7|=9$, and in this case, by $(P1)$, we have $|Y_2\cup Y_3\cup Y_4\cup Y_5\cup Y_7|=25$, which means that $K_{5,5}\subseteq \overline{G}[X\setminus \{x_1,x_6\}, Y]$. Therefore, the claim holds. 
 \end{proof}	
	Hence, consider $|B|$, By Claim \ref{c8}, we have $|B|\geq 3$. Now, we verify the following two claims.
\begin{claim}\label{c9}
	If $|B|= 3$, then we have $K_{5,5}\subseteq \overline{G}$.
\end{claim}	
\begin{proof}[Proof of Claim \ref{c9}]
	Suppose that $\{x_1,x_2, x_3\}= B$. W.l.g., assume that   $N_G(x_1)=Y_1=\{y_1,\ldots,y_7\}$, $ N_G(x_2)=Y_2=\{y_1,y_8,\ldots,y_{13}\}$, and  $ N_G(x_3)=Y_3=\{y_2, y_8, y_{14},\ldots,y_{18}\}$. Hence, it can be checked that $| \cup_{i=1}^{i=3} N_G(x_i)|=18$. If  $|(N_G(x)\cup N_G(x'))\cap (Y\setminus \cup_{i=1}^{i=3} N_G(x_i))|\leq 7$ for some $x,x' \in X\setminus B$, then it is clear that $K_{5,5}\subseteq \overline{G}$. So, let $|(N_G(x)\cup N_G(x'))\cap (\cup_{i=1}^{i=4} N_G(x_i))|\geq 8$ for each $x,x'\in X\setminus B$. Set $X'=X\setminus B$,  $Y'_1=\{y_3,\ldots,y_7\}$, $ Y'_2=\{y_9,\ldots,y_{13}\}$, and  $Y'_3=\{ y_{14},\ldots,y_{18}\}$. For each $i\in \{1,2,3\}$,  there are at least two members $x, x'$ of $X'$, so that $|N_G(x)\cap Y'_i|=|N_G(x')\cap Y_i'|=1$; otherwise  one can check  that  $K_{5,5}\subseteq \overline{G}[X\setminus \{x_i\}, Y'_i]$. Hence, if there is a member $x$ of $X'$, so that $|N_G(x)\cap(Y'_1\cup Y'_2\cup Y'_3 )|=3$, then  by pigeon-hole principle, there  is a member $x'$ of  $X'\setminus\{x\}$, such that $|N_G(x')\cap(Y'_1\cup Y'_2\cup Y'_3 )|=2$. Therefore, we have $|(N_G(x)\cap Y\setminus \cup_{i=1}^{i=3} N_G(x_i))\cup N_G(x')\cap (Y\setminus \cup_{i=1}^{i=3} N_G(x_i))|\leq 7$ and so $K_{5,5}\subseteq \overline{G}[B\cup \{x,x'\}, Y]$. Hence, let $|N_G(x)\cap(Y'_1\cup Y'_2\cup Y'_3 )|\leq 2$ for each $x\in X'$. Therefore, since $|B|=3$ and $|X'|=4$,  by pigeon-hole principle, it can be shown that there exist  two members $x,x'$ of $X'$, such that  $|N_G(x)\cap(Y'_1\cup Y'_2\cup Y'_3 )|=|N_G(x')\cap(Y'_1\cup Y'_2\cup Y'_3 )|=2$. W.l.g., assume that $x=x_4, x'=x_5$. Therefore, as  $|N_G(x)|\leq 6$ for each $x\in X''$, we have   $|N_G(x)\cap (Y\setminus \cup_{i=1}^{i=3} N_G(x_i))|\leq 4$ for each $x\in \{x_4,x_5\}$. If  either $|N_G(x)\cap N_G(x)\cap (Y\setminus \cup_{i=1}^{i=3} N_G(x_i))|\leq 3$ for one $x\in \{x_4,x_5\}$, or $|N_G(x_4)\cap N_G(x_5)\cap (Y\setminus \cup_{i=1}^{i=3} N_G(x_i))|=1$,  then we have $K_{5,5}\subseteq \overline{G}[B\cup \{x_4,x_5\}, Y]$. So, w.l.g. let $N_G(x_4)\cap (Y\setminus \cup_{i=1}^{i=3} N_G(x_i))=\{y_{19}, y_{20},y_{21}, y_{22}\}$ and $N_G(x_5)\cap (Y\setminus \cup_{i=1}^{i=3} N_G(x_i))=\{y_{23}, y_{24},y_{25}, y_{26}\}$. Now consider $x_6, x_7$. If  $|(N_G(x)\cap Y\setminus \cup_{i=1}^{i=3} N_G(x_i))|\leq 3$ for one $x\in \{x_6,x_7\}$, then the proof is the same. So, suppose that $|(N_G(x)\cap Y\setminus \cup_{i=1}^{i=3} N_G(x_i))|\geq 4$ for each $x\in \{x_6,x_7\}$. If  there exists a vertex  $x\in \{x_6,x_7\}$, such that $|(N_G(x)\cap Y\setminus \cup_{i=1}^{i=3} N_G(x_i))|= 4$ and $|(N_G(x)\cap N_G(x_j)\cap Y\setminus \cup_{i=1}^{i=3} N_G(x_i))|= 1$ for one $j\in \{4,5\}$, then the proof is the same. Hence, in any case, we have $|N_G(x_i)\cup \{y_{27}, y_{28}, y_{29}, y_{30}\}|\geq 3$, which means that $K_{2,2}\subseteq G$, which is  a contradiction and the claim holds. 
	 
\end{proof}	
\begin{claim}\label{c10}
	If $|B|= 4$, then we have $K_{5,5}\subseteq \overline{G}$.
\end{claim}	
\begin{proof}[Proof of Claim \ref{c10}]
	Suppose that $x_1,x_2, x_3,x_4\in B$ and $Y_i=N_G(x_i)$ for $i=1,2,3,4$. By Claim \ref{c1},  it can be checked that $| \cup_{i=1}^{i=4} N_G(x_i)|=22$. If $|N_G(x)\cap (Y\setminus \cup_{i=1}^{i=4} N_G(x_i))|\leq 3$ for one $x\in X\setminus B$, then  $K_{5,5}\subseteq \overline{G}$. So, let $|N_G(x)\cap (Y\setminus \cup_{i=1}^{i=4} N_G(x_i))|\geq 4$ for each $x\in X\setminus B$. Therefore, as $|X\setminus B|=3$, $|Y\setminus( \cup_{i=1}^{i=4} N_G(x_i))|=8$, and $|N_G(x)\cap (Y\setminus \cup_{i=1}^{i=4} N_G(x_i))|\geq 4$ for each $x\in X\setminus B$, in any case we have $K_{2,2}\subseteq G$, which is  a contradiction. 
 
\end{proof}	
	Therefore, by Claims \ref{c6},  \ref{c8},  \ref{c9}, and \ref{c10}, we may assume that $|B|\geq 5$,  and w.l.g., let  $X\setminus\{x_6,x_7\}\subseteq B$. Furthermore by Claims \ref{c7} and \ref{c0}, we have $| \cup_{i=1}^{i=5} N_G(x_i)|=7+6+5+4+3=25$ and so $K_{5,5}\subseteq \overline{G}[X\setminus \{x_6,x_7\}, Y]$.  Therefore, the proof of Case 2 is complete.
	
\bigskip 
\noindent \textbf{Case 3.} $ \Delta=8$. W.l.g., let $N_G(x_1)=Y_1=\{y_1,\ldots,y_8\}$. As $K_{2,2} \nsubseteq G$, we have $|N_G(x_i)\cap Y_1|\leq 1$. Hence, as $|Y_1|=8$ for each $x\in X\setminus\{x_1\}$, one can say that $K_{5,3}\subseteq \overline{G}[X\setminus \{x_1,x\},Y_1]$.  Let there is a member $x_2$ of $ X\setminus\{x_1\}$, so that $| N_G(x_2)\cap (Y\setminus Y_1)|=7$.  Therefore, as  $K_{2,2} \nsubseteq G$, we have $|N_G(x_i)\cap (N_G(x_2)\cap (Y\setminus Y_1)) |\leq 1$. Hence, since $|  N_G(x_2)\cap (Y\setminus Y_1)|=7$ and $|X|=7$, one can say that $K_{5,2}\subseteq \overline{G}[X\setminus \{x_1,x_2\},  N_G(x_2)\cap (Y\setminus Y_1)]$. So, $K_{5,5}\subseteq \overline{G}[X\setminus \{x_1,x_2\},  N_G(x_1)\cap  N_G(x_1)]$. Now, one can suppose that $| N_G(x)\cap (Y\setminus Y_1)|\leq 6$ for each $x\in X\setminus \{x_1\}$. Now, we are ready to prove the following claim.
\begin{claim}\label{c4}
	Suppose that	$| N_G(x)\cap (Y\setminus Y_1)|= 6$. If either $| N_G(x')\cap N_G(x)\cap (Y\setminus Y_1)|= 0$ for one $x'\in X\setminus \{x_1,x\}$, or $| N_G(x')\cap N_G(x'')\cap N_G(x)\cap (Y\setminus Y_1)|= 1$ for some $x',x''\in X\setminus \{x_1,x\}$,  then $K_{5,5}\subseteq \overline{G}$.
\end{claim}	
\begin{proof}[Proof of Claim \ref{c4}]
	W.l.g., let 	$| N_G(x_2)\cap (Y\setminus Y_1)|= 6$,  $ N_G(x_2)\cap (Y\setminus Y_1)=Y_2=\{y_9, y_{10} \ldots, y_{14}\}$, and   $| N_G(x_3)\cap Y_2|= 0$. Therefore,  Since $K_{2,2} \nsubseteq G$, we have $|N_G(x_i)\cap Y_1\cup Y_2|\leq 2$ for each $i\in \{4,5,6,7\}$. As $|Y_1\cup Y_2|=14$ and $|N_G(x_i)\cap Y_1\cup Y_2|\leq 2$, one can say that $| \cup_{i=3}^{i=7}( N_G(x_i)\cap Y_1\cup Y_2)|\leq 9$, which means that $K_{5,5}\subseteq \overline{G}[X\setminus \{x_1,x_2\},Y_1\cup Y_2]$. For the case that $| N_G(x')\cap N_G(x'')\cap (N_G(x)\cap (Y\setminus Y_1))|= 1$ for some $x',x''\in X\setminus \{x_1,x\}$, the proof is the same. Hence, the claim holds.
\end{proof}	
Set $A$ as follows:
\[A= \{x\in X\setminus \{x_1\}, ~~| N_G(x)\cap (Y\setminus Y_1)|= 6\}.\]
By considering $A$, we can prove the following claim.
\begin{claim}\label{c5}
	If $|A|\neq 0$, then we have $K_{5,5}\subseteq \overline{G}$.
\end{claim}	
\begin{proof}[Proof of Claim \ref{c5}]
	W.l.g., let 	$| N_G(x_2)\cap (Y\setminus Y_1)|= 6$ and $N_G(x_2)\cap (Y\setminus Y_1)=Y_2=\{y_9, y_{10} \ldots, y_{14}\}$.	If $|A|\geq 5$, then by Claim \ref{c4}, it can be said that $|\cup_{x_j\in A'} N_G(x_j)|\leq 25$ where $A'\subseteq A$ and $|A'|=5$. Hence, as $|Y|=30$,  we have $K_{5,5}\subseteq \overline{G}[A, Y]$.  So, suppose that $|A|\leq 4$. Assume that $|A|=i$,  and w.l.g. let $A=\{x_2,x_3,\ldots, x_{i+1}\}$, where $i\in \{1,2,3,4\}$. 
	
	By Claim \ref{c4}, for the case that $i=4$, we have $|\cup_{j=1}^{j=5} N_G(x_j)|= 8+6+5+4+3=26$.  W.l.g., assume that $\cup_{j=1}^{j=5} N_G(x_j)=Y\setminus \{y_{27}, y_{28}, y_{29}, y_{30}\}$. Hence, as $|Y|=30$ and  $K_{2,2}\nsubseteq G$, we have  $| N_G(x)\cap  \{y_{27}, y_{28}, y_{29}, y_{30}|\leq 2$ for at least one member of $\{x_6,x_7\}$. W.l.g., we may suppose that  $| N_G(x_6)\cap  \{y_{27}, y_{28}, y_{29}, y_{30}|\leq 2$.  Hence, we have $K_{6,2}\subseteq \overline{G}[X\setminus \{x_7\}, \{y,y'\}]$, where $y,y'\in   \{y_{27}, y_{28}, y_{29}, y_{30}\}$. W.l.g., assume that $y=y_{29} , y'=y_{30}$. Therefore, as $| N_G(x_i)\cap Y_1|\leq 1$ for each $x_i\in X\setminus \{x_1\}$, we have  $K_{5,3}\subseteq \overline{G}[X\setminus \{x_1,x_7\},Y_1]$.  Therefore  $K_{5,5}\subseteq \overline{G}[X\setminus \{x_1,x_7\},Y_1\cup\{y_{29},y_{30}\}]$.  
	
	For the case that $i=3$, by Claim \ref{c4},  we have $|\cup_{j=1}^{j=4} N_G(x_j)|= 8+6+5+4=23$. W.l.g., assume that $\cup_{j=1}^{j=4} N_G(x_j)=Y'=\{y_1\ldots y_{23}\}$. Hence, if there exists one $i\in \{5,6,7\}$ such that $| N_G(x_i)\cap Y\setminus Y'|\leq 2$, then the proof is the same. So suppose that  for each $i\in \{5,6,7\}$ we have $| N_G(x_i)\cap Y\setminus Y'|\geq 3$. Now as $K_{2,2}\nsubseteq G$, one can checked that there exist at least two member $x_5, x_6$ of $\{x_5,x_6,x_7\}$, such that $| (N_G(x_5)\cup N_G(x_6))  \cap (Y\setminus Y')|\leq 5$. Hence the proof is the same as the case that $i=4$.
	
	For the case that $i=2$, by Claim \ref{c4},  we have $|\cup_{j=1}^{j=3} N_G(x_j)|= 8+6+5=19$. Therefore, using the fact that $|A|=2$ and  Claim \ref{c4}, we have $|\cup_{j=1}^{j=5} N_G(x_j)|\leq  25$.  Hence, the proof is the same.	
	
	So, suppose that $|A|=1$.  Hence, by Claim \ref{c4},  we have $|\cup_{j=1}^{j=2} N_G(x_j)|= 8+6=14$.  Hence, as $|Y|=30$ and by Claim \ref{c4}, we have $| N_G(x_j)\cap N_G(x_2)\cap (Y\setminus Y_1)|= 1$ for each  $x_j\in X'=X\setminus A=\{x_3,\ldots, x_7\}$. Therefore,  we have $| N_G(x)\cap  \{y_{15}, \ldots  y_{30}|\leq 4$ for each $x\in X'$. Set $Y'=Y\setminus (Y_1\cup Y_2)$. As $|N_G(x_1)\cap Y'|\leq 4$, if  there is one member $x_3$ of $X'$, so that $|N_G(x_i)\cap Y'|\leq 3$, then  one can  check that	$K_{5,5}\subseteq \overline{G}[X\setminus \{x_6,x_7\}, Y]$. So, suppose that $|N_G(x)\cap Y'|= 4$. Therefore as $K_{2,2}\nsubseteq G$,  $|Y'|=16$, and $|X'|=5$, it is easy to say that there are two members $x_3,x_4$ of $X'$, such that  $| N_G(x_3)\cap N_G(x_4)\cap Y'|= 1$. Hence we have $| (N_G(x_3)\cap Y') \cup (N_G(x_4)\cap Y')|= 7$. Therefore 	$K_{5,5}\subseteq \overline{G}[X\setminus \{x_6,x_7\}, Y]$. Which means that the proof is complete.
\end{proof}	
Now, by Claim \ref{c5}, let $|A|=0$, that is $| N_G(x)\cap (Y\setminus Y_1)|\leq  5$ for each $x\in X\setminus\{x_1\}$. If 	$| N_G(x)\cap (Y\setminus Y_1)|\leq  4$ for each $x\in X\setminus\{x_1\}$, then $|\cup_{j=2}^{j=6} N_G(x_j)|\leq 25$, which means that	$K_{5,5}\subseteq \overline{G}$. So, w.l.g. let  $| Y_2=N_G(x_2)\cap (Y\setminus Y_1)|= 5$ and let $Y_2=\{y_9,\ldots, y_{13}\}$. As $K_{2,2} \nsubseteq G$, we have $|N_G(x_i)\cap Y_2|\leq 1$ for each $i\neq 1$. Therefore, as $|Y_2|=5$, one can assume that  $|N_G(x_i)\cap Y_2|=1$ for at least two members of $X\setminus \{x_1,x_2\}$; otherwise $K_{5,5}\subseteq \overline{G}[X\setminus \{x_2\}, Y_2]$ and $K_{5,5}\subseteq \overline{G}[X\setminus \{x_1,x_2\},Y_1\cup Y_2]$. W.l.g., assume that  $|N_G(x_i)\cap Y_2|=1$ for each $x_i\in X'''=\{x_3,x_4\}$. If either 	$| N_G(x)\cap (Y\setminus Y_1\cup Y_2)|\leq  3$ for at least one  $x\in X'''$, or $| N_G(x_3)\cap N_G(x_4)\cap Y_2|=1$, then $|\cup_{j=1}^{j=5} N_G(x_j)|\leq 25$, which means that	$K_{5,5}\subseteq \overline{G}$. Now, for $i=3,4$, w.l.g.  suppose that  $| N_G(x_i)\cap (Y\setminus Y_1\cup Y_2)|= 4$ and $| N_G(x_3)\cap N_G(x_4)\cap  Y_2|=0$. W.l.g., let $N_G(x_3)\cap (Y\setminus Y_1\cup Y_2)=Y_3=\{y_{14},\ldots, y_{17}\}$ and $N_G(x_4)\cap (Y\setminus Y_1\cup Y_2)=Y_4=\{y_{18},\ldots, y_{21}\}$. If there exists a member $x$ of $\{x_5,x_6,x_7\}$, such that   $|N_G(x)\cap \{y_{22}, \ldots, y_{30} \}|\leq 4$, then we have  $K_{5,5}\subseteq \overline{G}[\{x_1,x_2, x_3,x_4, x\},Y]$. Therefore,  assume that
  $|N_G(x)\cap \{y_{22}, \ldots, y_{30} \}|\geq 5$ for each $x\in \{x_5,x_6,x_7\}$. In any case  $K_{2,2}\subseteq G$, which is  a contradiction. So, the proof of the Case 3 is complete.
	
	Hence, by Cases 1, 2, 3, the proof of $(II)$ is complete. 	Thus,  $BR_7(K_{2,2}, K_{5,5}) \leq 30$ and  by $(I)$, $BR_7(K_{2,2}, K_{5,5})= 30$. This also implies that $BR_8(K_{2,2}, K_{5,5})=30$. Hence the theorem holds.
\end{proof}

	\begin{proof}[\bf Proof of Theorem \ref{M.th}]
	By combining Theorems  \ref{th0}, \ref{th2},  and \ref{th3},  we conclude that the proof of Theorem \ref{M.th} is complete.
	\end{proof}
 
	\section{Declarations}
{\bf Conflict of Interest:} On behalf of all authors, the corresponding author states	that there is no conflict of interest.

	{\bf Data Availability Statement:}No data were generated or used in the preparation of this paper.
	
	\bibliographystyle{spmpsci} 
	\bibliography{BI}

\begin{thebibliography}{10}
\providecommand{\url}[1]{{#1}}
\providecommand{\urlprefix}{URL }
\expandafter\ifx\csname urlstyle\endcsname\relax
  \providecommand{\doi}[1]{DOI~\discretionary{}{}{}#1}\else
  \providecommand{\doi}{DOI~\discretionary{}{}{}\begingroup
  \urlstyle{rm}\Url}\fi

\bibitem{beinere1976bipartite}
BEINERE, L., LW, B., AJ, S.: On a bipartite form of the ramsey problem.  (1976)

\bibitem{bi2018another}
Bi, Z., Chartrand, G., Zhang, P.: Another view of bipartite ramsey numbers.
\newblock Discussiones Mathematicae: Graph Theory \textbf{38}(2) (2018)

\bibitem{bi2019new}
Bi, Z., Chartrand, G., Zhang, P.: A new view of bipartite ramsey numbers.
\newblock J. Combin. Math. Combin. Comput \textbf{108}, 193--203 (2019)

\bibitem{bucic20193}
Buci{\'c}, M., Letzter, S., Sudakov, B.: 3-color bipartite ramsey number of
  cycles and paths.
\newblock Journal of Graph Theory \textbf{92}(4), 445--459 (2019)

\bibitem{bucic2019multicolour}
Bucic, M., Letzter, S., Sudakov, B.: Multicolour bipartite ramsey number of
  paths.
\newblock The Electronic Journal of Combinatorics pp. P3--60 (2019)

\bibitem{chartrand2021new}
Chartrand, G., Zhang, P.: New directions in ramsey theory.
\newblock Discrete Math. Lett \textbf{6}, 84--96 (2021)

\bibitem{collins2016zarankiewicz}
Collins, A.F., Riasanovsky, A.W., Wallace, J.C., Radziszowski, S.: Zarankiewicz
  numbers and bipartite ramsey numbers.
\newblock Journal of Algorithms and Computation \textbf{47}, 63--78 (2016)

\bibitem{gholami2021bipartite}
Gholami, M., Rowshan, Y.: The bipartite ramsey numbers $ br (c\_8, c\_
  $\{$2n$\}$) $.
\newblock arXiv preprint arXiv:2108.02630  (2021)

\bibitem{goddard2000bipartite}
Goddard, W., Henning, M.A., Oellermann, O.R.: Bipartite ramsey numbers and
  zarankiewicz numbers.
\newblock Discrete Mathematics \textbf{219}(1-3), 85--95 (2000)

\bibitem{goedgebeur2022new}
Goedgebeur, J., Van~Overberghe, S.: New bounds for ramsey numbers r (kk- e, kl-
  e).
\newblock Discrete Applied Mathematics \textbf{307}, 212--221 (2022)

\bibitem{gyarfas1973ramsey}
Gy{\'a}rf{\'a}s, A., Lehel, J.: A ramsey-type problem in directed and bipartite
  graphs.
\newblock Period. Math. Hungar \textbf{3}(3-4), 299--304 (1973)

\bibitem{hatala2021new}
Hatala, I., H{\'e}ger, T., Mattheus, S.: New values for the bipartite ramsey
  number of the four-cycle versus stars.
\newblock Discrete Mathematics \textbf{344}(5), 112320 (2021)

\bibitem{hattingh1998star}
Hattingh, J.H., Henning, M.A.: Star-path bipartite ramsey numbers.
\newblock Discrete Mathematics \textbf{185}(1-3), 255--258 (1998)

\bibitem{kamranian2022star}
Kamranian, A., Raeisi, G.: On the star-critical ramsey number of a forest
  versus complete graphs.
\newblock Iranian Journal of Science and Technology, Transactions A: Science
  \textbf{46}(2), 499--505 (2022)

\bibitem{lakshmi2020three}
Lakshmi, R., Sindhu, D.: Three-colour bipartite ramsey number r\_b (g\_1, g\_2,
  p\_3).
\newblock Electronic Journal of Graph Theory and Applications (EJGTA)
  \textbf{8}(1), 195--204 (2020)

\bibitem{raeisi2015star}
Raeisi, G.: Star-path and star-stripe bipartite ramsey numbers in
  multicoloring.
\newblock Transactions on Combinatorics \textbf{4}(3), 37--42 (2015)

\bibitem{rowshan2}
Rowshan, Y.: The m-bipartite ramsey number $br_m(h_1, h_2)$.
\newblock Discussiones Mathematicae Graph Theory \textbf{0} (2022)

\bibitem{2022arXiv220112844R}
{Rowshan}, Y., {Gholami}, M.: {Another view of Bipartite Ramsey numbers}.
\newblock arXiv e-prints arXiv:2201.12844 (2022)

\bibitem{rowshan2023multicolor}
Rowshan, Y., Gholami, M.: Multicolor bipartite ramsey numbers for paths,
  cycles, and stripes.
\newblock Computational and Applied Mathematics \textbf{42}(1), 25 (2023)

\bibitem{rowshan2021size}
Rowshan, Y., Gholami, M., Shateyi, S.: The size, multipartite ramsey numbers
  for nk2 versus path--path and cycle.
\newblock Mathematics \textbf{9}(7), 764 (2021)

\bibitem{rowshan2022proof}
Rowshan, Y., Gholami, M., Shateyi, S.: A proof of a conjecture on bipartite
  ramsey numbers b (2, 2, 3).
\newblock Mathematics \textbf{10}(5), 701 (2022)

\bibitem{wang2021bipartite}
Wang, Y., Li, Y., Li, Y.: Bipartite ramsey numbers of kt, s in many colors.
\newblock Applied Mathematics and Computation \textbf{404}, 126220 (2021)

\end{thebibliography}
\end{document}